\documentclass{amsart}
\usepackage{amssymb,amsmath,amsthm,amsfonts}
\usepackage{framed,comment}
\usepackage[pdftex]{graphicx}
\def\R {\mathbb{R}}
\def\C {\mathcal{C}}
\def\N {\mathbb{N}}
\def\S {\mathbb{S}}

\def\bv {\mathbf{v}}
\def\bu {\mathbf{u}}
\def\bw {\mathbf{w}}

\def\eps{\varepsilon}
\def\dist{{\rm dist}}

\newcommand{\loc}{\mathrm{loc}}
\newcommand{\alphaopt}{\alpha_\mathrm{opt}}

\renewcommand{\div}{\mathrm{div}}

\newcommand{\problem}[1] {(P)_{#1}}

\newcommand{\de}[1] {\mathrm{d} #1}

\DeclareMathOperator*{\pv}{pv}
\DeclareMathOperator*{\tsum}{\textstyle{\sum}}

\DeclareMathOperator{\supp}{supp}

\newtheorem{proposition}{Proposition}[section]
\newtheorem{theorem}[proposition]{Theorem}
\newtheorem{corollary}[proposition]{Corollary}
\newtheorem{lemma}[proposition]{Lemma}

\theoremstyle{definition}
\newtheorem{definition}[proposition]{Definition}
\newtheorem{remark}[proposition]{Remark}
\numberwithin{equation}{section}

\title[Lotka-Volterra competition vs fractional diffusion]{Strong competition versus fractional diffusion:\\
the case of Lotka-Volterra interaction}

\author{Gianmaria Verzini }
\email{gianmaria.verzini@polimi.it}
\address{Dipartimento di Matematica, Politecnico di Milano, p.za Leonardo da
Vinci 32,  20133 Milano, Italy}

\author{Alessandro Zilio}
\email{alessandro.zilio@polimi.it}
\address{Dipartimento di Matematica, Politecnico di Milano, p.za Leonardo da
Vinci 32,  20133 Milano, Italy}

\thanks{Work partially supported by the PRIN2009 grant ``Critical Point Theory and
Perturbative Methods for Nonlinear Differential Equations''.}

\begin{document}

\begin{abstract}
We consider a system of differential equations with nonlinear Steklov boundary conditions, related to the fractional problem
\[
(-\Delta)^s \bu = \mathbf{f}(x,\bu) - \beta u_i^p \sum_{j\neq i} a_{ij} u_j^p,
\]
where $\bu=(u_1,\dots,u_k)$, $s\in(0,1)$, $p>0$, $a_{ij}>0$ and $\beta>0$. When $k=2$
we develop a quasi-optimal regularity theory in $\C^{0,\alpha}$, uniformly w.r.t. $\beta$,
for every $\alpha < \alphaopt=\min(1,2s)$; moreover we show that the traces of the limiting
profiles as $\beta\to+\infty$ are Lipschitz continuous and segregated.
Such results are extended to the case of $k\geq3$ densities,
with some restrictions on $s$, $p$ and $a_{ij}$.
\end{abstract}

\maketitle

\section{Introduction}

%
%
%
%

Let us consider the following stationary differential system,
involving $k\geq2$ non negative densities $u_i$ which are subject to diffusion, reaction and
competition
\begin{equation}\label{eqn: intro_gen}
(-\Delta)^s u_i = f_{i}(x,u_i) - \beta u_i^p \sum_{j\neq i} a_{ij} u_j^q,
\end{equation}
settled in $H^{s}(\R^N)$, $N\geq1$, or in a bounded domain with suitable boundary conditions. In this system,
different ranges of the parameter $s$ allow to model the brownian diffusion ($s=1$), as well as the fractional one
($0<s<1$), which arises whenever the underlying Gaussian process is replaced by
the Levy one, in order to allow discontinuous random walks. In the latter case, the nonlocal operator
\[
(-\Delta)^s u (x) = c_{N,s} \pv\int_{\R^N} \frac{u(x)-u(\xi)}{|x-\xi|^{N+2s}}\,\de\xi
\]
denotes the $s$-power of the laplacian. Furthermore, the competitive nature of the
interaction is driven by the positivity of the parameters $\beta$, $p$, $q$ and $a_{ij}$, $1\leq i,j\leq k$.
Among others, two types of competition are particularly relevant in the applications:
\begin{itemize}
 \item the case $p=q=1$, that is the \emph{Lotka-Volterra} type competition, which is widely used in population dynamics and ecology;
 \item the case $p=1$, $q=2$ (and $a_{ij}=a_{ji}$), which turns \eqref{eqn: intro_gen} into the \emph{Gross-Pitaevskii} system: this system arises in the search of solitary waves associated to the cubic Schr\"odinger system, which is commonly accepted as a model for Bose-Einstein condensation in multiple states, and often used also in nonlinear optics. In great contrast with the Lotka-Volterra one, this system has a \emph{variational} structure.
\end{itemize}
In the study of \eqref{eqn: intro_gen}, a peculiar issue is the analysis of the behavior of the densities in the
case of \emph{strong competition}, i.e. when $\beta\to+\infty$. In such situation, one expects the formation of self-organized patterns, in which the limiting densities are spatially segregated, and the natural questions regard
\emph{a}) the common regularity shared by families of solutions, uniformly in $\beta$ and \emph{b}) the properties of the limiting segregated profile. In facing such questions, typical tools are the blow-up technique and the monotonicity formulae of Alt-Caffarelli-Friedman and of Almgren type.

After \cite{clll,ctvNehari,ctvOptimal}, the case $s=1$ of standard diffusion has been extensively
studied in the last decade. In particular it is known that, both in the case of Lotka-Volterra
competition \cite{ctv,ckl} and in the variational one \cite{caflin,nttv},
each family of solutions which share a common uniform bounds in the $L^{\infty}$ norm is precompact
in the topology of $H^1 \cap \C^{0,\alpha}$ for every $\alpha < 1$; we highlight that this result is quasi-optimal, in the sense that $\alpha=1$ is the maximal common regularity allowed for this problem.
Furthermore, the limiting profiles (as $\beta\to+\infty$) are solutions of the segregated system
\begin{equation}\label{eqn: seg sys lap}
    u_i \left(- \Delta u_i - f_i(x,u_i) \right) = 0, \qquad u_i u_j = 0\text{ for }j\neq i,
\end{equation}
they are Lipschitz continuous, and they obey to a weak reflection law which roughly says that,
on the free boundary separating two components, the corresponding gradients are equal in magnitude (up to suitable scaling factors depending on the matrix $(a_{ij})$) and opposite in direction \cite{tt}. Remarkably, such law is the same for both types of competition \cite{dwz3}. For some related results, in the case of standard diffusion, we
also refer to \cite{dwz1,nttv2} and references therein.

Coming to the anomalous diffusion case $s\in(0,1)$, for the moment only the competition of
Gross-Pitaevskii type has been considered in the literature. In such framework, the results above
were recently generalized \cite{tvz1,tvz2,z_tesi} in the following sense: $L^\infty$ uniform bounds
imply uniform bounds in $H^s \cap \C^{0,\alpha}$ (for a suitable extension problem), for every
$\alpha<\alphaopt^{\text{GP}}(s)$. Here the optimal exponent
\[
\alphaopt^{\mathrm{GP}}(s)=s,
\]
at least when $0<s\leq 1/2$; for $1/2<s< 1$ we could only show that $\alphaopt^{\text{GP}}(s)
\geq 2s-1$, because of the lack of a clean-up lemma appropriate to exclude self-segregation; see
\cite{tvz2} for further details. In any case, this result agrees with the one holding for the
standard Laplace operator, since $\alphaopt^{\text{GP}}(1)=1$. Moreover the limiting
profiles satisfy a natural extension to the fractional setting of the system
\eqref{eqn: seg sys lap}, that is
\begin{equation}\label{eqn: seg sys flap_gs}
    u_i \left( (- \Delta )^s u_i - f_i(x,u_i) \right) = 0, \qquad u_i u_j = 0\text{ for }j\neq i,
\end{equation}
and the validity of an Almgren monotonicity formula across the free boundary ensures a
reflection property, as in the case $s=1$.

Under the perspective just described, in this paper we address the study of system
\eqref{eqn: intro_gen} in the case $s\in(0,1)$ and
\[
p=q>0.
\]
We remark that such range of parameters not only includes the Lotka-Volterra competition,
but it is of interest also in the complementary case $p\neq 1$. Indeed, in the case of $k=2$
components, such competition appears in the modeling of diffusion flames \cite{crs}, while in the
general case the change of variables $U_i=u_i^p$ turns system \eqref{eqn: intro_gen} into the
one for competing densities subject to fast fractional diffusion (when $p>1$), or to fractional
diffusion in a porous medium (when $p<1$) \cite{dqrv,bv}.

As in \cite{tvz1,tvz2}, we state our results for a localized extension problem \cite{cs}
related to the nonlocal system \eqref{eqn: intro_gen}, namely the problem
\[
    \begin{cases}
    L_a  v_{i} = 0 & \text{in } B^+_1\\
    \partial^a_{\nu} v_{i} = f_{i,\beta}(x, v_{1}, \dots, v_{k}) - \beta v_{i}^p \tsum_{j \neq i} a_{ij} v_{j}^p & \text{on } \partial^0 B^+_1,
    \end{cases} \leqno \problem{\beta}
\]
where we adopt the standard notations $\R^{N+1}_+:=\{X=(x,y)\in\R^N\times\R:y>0\}$,
$B^+_r:=B_r\cap\{y>0\}\subset\R^{N+1}_+$, $\partial^+ B^+_r := \partial B_r\cap\{y>0\}$,
$\partial^0 B^+_r:= B_r\cap\{y=0\}\subset\R^{N}$, and
\[
    L_a v := -\div\left(|y|^a \nabla v\right), \qquad
    \partial^a_{\nu} v := \lim_{y \to 0^+} - y^a \partial_{y} v,
\]
for $a:=1-2s\in(-1,1)$. This last condition insures that the weight $y^a$ belongs to the
Muckenhoupt $A_2$-class \cite{kufner}, so that a weak version of $\problem{\beta}$ can
be formulated in the Hilbert space
\[
    H^{1;a}(B^+_1) := \left\{v : \int_{B^+_1} y^a \left(|v|^2 + |\nabla v|^2 \right)\, \de{x}\de{y} < \infty \right\}.
\]
Our first main results concern the full quasi-optimal theory in the case of two densities.
\begin{theorem}\label{thm: intro_hold2}
Let $p>0$, $a_{ij}>0$ for any $j\neq i$, and the reaction terms $f_{i,\beta}$ be continuous and map
bounded sets into bounded sets, uniformly w.r.t. $\beta>0$.

If $k=2$ then, for every
\[
\alpha<\alphaopt(s)=\alphaopt^{\mathrm{LV}}(s):=\min(2s,1)
\]
and $\bar m>0$, there exists a constant $C=C(\alpha,\bar m)$ independent of $\beta$ such that
\[
 \| \bv_{\beta} \|_{L^{\infty}(B^+)} \leq \bar m
 \quad\implies\quad
 \| \bv_\beta\|_{\C^{0,\alpha}\left(\overline{B^+_{1/2}}\right)} \leq C,
\]
for every $\bv_\beta=(v_{1,\beta},v_{2,\beta})$ nonnegative solution of problem $\problem{\beta}$.

Furthermore, any sequence of uniformly bounded, nonnegative solutions
$\{(v_{1,\beta_n},v_{2,\beta_n})\}_n$, with $\beta_n\to\infty$,
converges (up to subsequences) in $\left(H^{1;a}\cap \C^{0,\alpha}\right)\left(\overline{B^+_{1/2}}\right)$
to a limiting profile $(v_1,v_2)$.
\end{theorem}
\begin{figure}[!h]
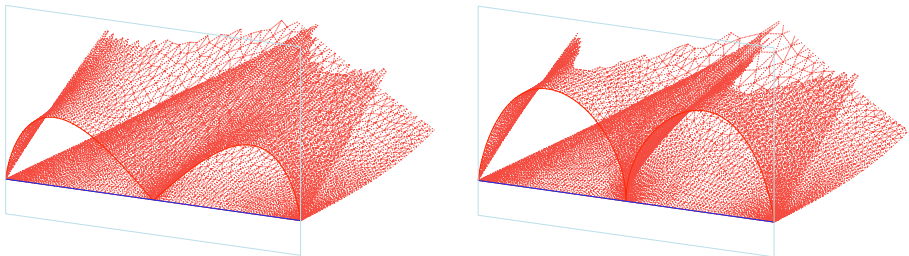

\centering
\includegraphics[width=.46\textwidth]{LV_interface.pdf}
\quad
\includegraphics[width=.46\textwidth]{BE_interface.pdf}
\caption{On the left, a numerical approximation of a limiting profile for problem 
$\problem{\beta}$ with Lotka-Volterra competition $p=q=1$ and $s=1/2$, for which Lipschitz 
continuity of the segregated traces is shown in Theorem \ref{thm: intro_lip2}. On the right,
the simulation for the analogous problem with Gross-Pitaevskii competition $p=1$, $q=2$, 
which optimal regularity, according to \cite[Theorem 1.2]{tvz1}, is only $\C^{0,1/2}$.
\label{fig: simulaz}}
\end{figure}
\begin{theorem}\label{thm: intro_lip2}
Under the assumption of the previous theorem, let furthermore $f_{i,\beta}\to f_i$ as $\beta\to\infty$, uniformly on compact sets, with $f_i$ Lipschitz continuous. For any limiting  profile $(v_1,v_2)$:
\begin{itemize}
 \item $v_1(x,0)$, $v_2(x,0)$ are Lipschitz continuous (optimal regularity of the traces);
 \item $v_1(x,0)\cdot v_2(x,0)=0$ (boundary segregation condition);
 \item $L_a v_1 = L_a v_2 = 0$ for $y>0$;
 \item $\partial_\nu^a (a_{21}v_1 - a_{12}v_2)=a_{21}f_1 - a_{12}f_2$  for $y=0$.
\end{itemize}
\end{theorem}
\begin{remark}
In the previous results $B^+_{1/2}$ can be replaced by any domain $\Omega\cap\{y>0\}$, where
$\overline{\Omega} \subset B_1$.
\end{remark}
\begin{remark}
Throughout this paper, we restrict our discussion to nonnegative solutions only to avoid technicalities.
Reasoning as in \cite{tvz1}, also changing sign solutions can be considered, once the competition is suitably
extended to negative densities.
\end{remark}
\begin{remark}
The upper bound $\alpha=2s$ for the regularity of the functions $v_{i,\beta}$ can not be removed:
indeed, from any solution of $\problem{\beta}$ we can construct another solution having
$(k+1)$ components, by defining $v_{k+1,\beta}(x,y)=y^{2s}$, $f_{k+1,\beta}\equiv-2s$. One
may possibly expect to be able to remove such threshold by considering only the regularity of the traces
$v_{i,\beta}(x,0)$, as suggested by Theorem \ref{thm: intro_lip2}.

On the other hand, the Lipschitz regularity is the natural one, at least for the traces, since the
last condition in Theorem \ref{thm: intro_lip2} implies that $v_i(x,0)$ are (proportional to)
the positive/negative parts of a
regular function.
\end{remark}
Next, we address the case of $k\geq3$ densities.
\begin{theorem}\label{thm: intro_holdk}
Let $k\geq3$. Then there exists $\alpha^*>0$ such that Theorem \ref{thm: intro_hold2} holds for any $\alpha<\alpha^*$, under the further assumption that
\[
\text{either }p\geq 1
\qquad
\text{ or }
\qquad a_{ij}=1\text{ for every }j\neq i.
\]
Furthermore, if $a_{ij}=1$,
\[
\alpha^*=\alphaopt(s)=\min(2s,1)
\]
whenever $s=1/2$ or $s\in(0,1/4)$.
\end{theorem}
%
Even though we can show quasi-optimality only in some cases, the above regularity result is sufficient to conclude
that, as $\beta\to\infty$, solutions of $\problem{\beta}$ accumulate to limiting profiles $v_i$ which properties,
apart from optimal regularity, are analogous to those described in Theorem \ref{thm: intro_lip2} for the case $k=2$
(see Section \ref{sec: lim_prof} for further details). In particular, going back to the segregated traces $u_i(x)=v_i(x,0)$, we can show that
\begin{equation}\label{eqn: seg sys flap_lv}
    u_i \left[ (- \Delta )^s \left(u_i-\tsum_{j\neq i} \frac{a_{ij}}{a_{ji}}u_j\right) -
    \left(f_i-\tsum_{j\neq i} \frac{a_{ij}}{a_{ji}}f_j\right) \right] = 0, \quad u_i u_j = 0\text{ for }j\neq i.
\end{equation}
Comparing with equation \eqref{eqn: seg sys flap_gs}, we see that, if $s<1$, the Gross-Pitaevskii competition and the Lotka-Volterra one exhibit deep differences not only from the point of view of the optimal regularity exponent, but
also from that of the differential equations satisfied by the segregated limiting profiles.
This is in great contrast with the case $s=1$ where, as we already mentioned, the two competitions can not
be distinguished from each other in the limit. Such feature is caused by the non local nature of the diffusion operators: indeed equation \eqref{eqn: seg sys flap_lv} can not be directly reduced to \eqref{eqn: seg sys flap_gs}, since in the set $\{u_i=0\}$ the corresponding fractional laplacian does not necessarily vanish. Nonetheless,
letting $s \to 1^-$, we recover the local nature of the equation: as a consequence
\[
   u_i  (- \Delta )^s \left(u_i-\sum_{j\neq i} \frac{a_{ij}}{a_{ji}}u_j\right) \to u_i(- \Delta u_i),
\]
so that equation \eqref{eqn: seg sys lap} arises also in this case.

To conclude, we mention that the equations just discussed --or, better, the corresponding ones for the extensions $v_i$-- can be used to obtain further regularity for the limiting profiles, also in the case $a_{ij}\neq1$. In particular, we have the following result.
\begin{theorem}\label{thm: intro_lastholdk}
Let $k\geq3$, $s=1/2$, $p>1$. If furthermore $f_i(x,t_1,\dots,t_k)=0$ for $|(t_1,\dots,t_k)|$ small then
every segregated limiting profile $v_i$ is $\C^{0,\alpha}$, for every $\alpha<1$.
\end{theorem}
\begin{remark}
Collecting together the results of Theorems \ref{thm: intro_holdk} and \ref{thm: intro_lastholdk}, we have that
for $s=1/2$ the limiting profiles are $\C^{0,\alpha}$, for every $\alpha<1$, when either $a_{ij}=1$ or $p>1$.
Since for $s=1/2$ we have that $L_a=(-\Delta)$, one may then try to apply the arguments contained in
\cite[Section 2]{ckl} (see also \cite[Section 5]{acf}). This should eventually imply that the traces of the limiting profiles are indeed Lipschitz continuous.
\end{remark}

\section{Preliminary results}
We devote this section to some results concerning the operator $L_a$ and solutions to some associated differential problem. Most of such results already appeared, even if in slightly different form, in the literature. The interested reader may refer to \cite{css,tvz1,tvz2} for further details.
\begin{lemma}[{\cite[Lemma 2.7]{css}}]\label{lem: first collection_s}
If $v$ is a non constant, global solution of $L_av=0$ in $\R^{N+1}$, with the property that
\[
    |v(X)| \leq C\left(1 + |X|^{\gamma}\right),
\]
then $\gamma \geq \min(2s,1)$. If furthermore $v(x,-y)=v(x,y)$ then $\gamma \geq 1$ (and $v$ is a polynomial).
\end{lemma}

\begin{lemma}[{\cite[Proposition 11]{tvz2}}]\label{lem: prescribed constant normal derivative_s}
Let $v$ satisfy
\begin{equation*}
    \begin{cases}
    L_a v = 0 & \text{in } \R^{N+1}_+ \\
    \partial_{\nu}^a v = \lambda  & \text{on } \R^{N}
    \end{cases}
\end{equation*}
for some $\lambda \in \R$, and
\[
    |v(X)| \leq C\left(1 + |X|^{\gamma}\right),
\]
for some $0\leq \gamma < \min(2s,1)$. Then $v$ is constant.
\end{lemma}
The two last results we need are based on the following comparison principle.
\begin{lemma}[Comparison principle]
Let $u,v \in H^{1;a}(B^+)$ satisfy
\[
    \begin{cases}
    L_a u \leq 0            &\text{in $B_1^+$}\\
    \partial_{\nu}^a u \leq -M u^p + \delta  &\text{on $\partial^0 B_1^+$},
    \end{cases}
\qquad
    \begin{cases}
    L_a v \geq 0            &\text{in $B_1^+$}\\
    \partial_{\nu}^a v \geq -M v^p + \delta  &\text{on $\partial^0 B_1^+$},
   \end{cases}
\]
respectively. Then $u \leq v$ on $\partial^+ B_1^+$ implies $u \leq v$ on $B_1^+$.
\end{lemma}
\begin{proof}
Letting $w = u -v$, we obtain that $w$ is a solution to
\[
    \begin{cases}
    L_a w \leq 0            &\text{in $B_1^+$}\\
    \partial_{\nu}^a w \leq -M (u^p-v^p)  &\text{on $\partial^0 B_1^+$}\\
    w \leq 0  &\text{on $\partial^+ B_1^+$}.
    \end{cases}
\]
Testing the equation with $w^+$ and recalling that $p>0$ we find
\[
    \int_{B_1^+} y^a |\nabla w^+|^2 \mathrm{d}x\mathrm{d}y \leq - M \int_{\partial^0 B^+} \frac{u^p-v^p}{u-v} (w^+)^2 \mathrm{d}x \leq 0.
    \qedhere
\]
\end{proof}
\begin{lemma}\label{lem: decay with perturbations_s}
Let $M>0$ be any large constant and $\delta>0$ be fixed and let $h \in L^{\infty}(\partial^0B_1^+)$ with $\|h\|_{L^{\infty}} \leq \delta$. Any $v \in H^{1;a}(B_1^+)$ non negative solution to
\[
    \begin{cases}
    L_a v \leq 0            &\text{in $B_1^+$}\\
    \partial_{\nu}^a v \leq -M v^p + h  &\text{on $\partial^0 B_1^+$}
    \end{cases}
\]
verifies
\[
    \sup_{\partial^0 B_{1/2}^+} v \leq \frac{1+\delta}{M^{1/p}} \sup_{\partial^+B^+_1} v.
\]
\end{lemma}

\begin{proof}[Sketch of proof]
The proof is similar to the one of \cite[Lemma 3.2]{tvz2}, the only difference being in the choice of the supersolution. For $a \in (-1,1)$ and $p > 0$ fixed, let $b = 1 + (1-a)/p >1$
and $f \in AC(\R) \cap \C^{\infty}(\R)$ be defined as
\[
    f(x) = c\int_{-\infty}^{x} \frac{1}{(1 + t^2)^{b/2}} \de t,
\]
where $c$ is chosen in such a way that $f(+\infty)=1$.
Then, for some $C>0$, the estimate
\[
    (-\Delta)^s f (x) \geq -C f(x)^{p}
\]
holds for any $x < 0$. For $M>0$, the function $f_M(x) := f(M^{1/(2s)}x)$ satisfies
\[
    (-\Delta)^s f_M (x) = M^{2s/(2s)} \left[(-\Delta)^s f\right] (M^{1/(2s)}x) \geq -CM f_M^p(x).
\]
Therefore, if we let
\[
    g_M(x) := f_M(x-1) + f_M(-x-1)
\]
then
\[
    \left(f_M^p(x-1) + f_M^p(-x-1)\right)^{1/p} \leq c_p g_M,
\]
for some $c_p>0$. It follows that, for any $M>0$, it holds
\[
    \begin{cases}
    (-\Delta)^s g_M (x) \geq -CM g_M^p(x) &\text{ in $(-1,1)$ }\\
    g_M(x) \geq \frac12  &\text{ in $\R \setminus (-1,1)$ }\\
    g_M(x) \leq C M^{-1/p}  &\text{ in $\left(-\frac12,\frac12\right)$. }
    \end{cases}
\]
The lemma follows by comparison between $v$ and the supersolution (see \cite{cs})
\[
    w_{\delta} := \delta \frac{1}{M^{1/p}} + \int_{\R} y^{1-a}\frac{ g_{M} (x-\xi)}{(\xi^2+y^2)^{1-a/2}} \mathrm{d}\xi. \qedhere
\]
\end{proof}
\begin{lemma}\label{lem: global eigenfunction_s}
Let $\lambda>0$ and $v\in H^{1;a}_{\loc}(\overline{\R^{N+1}_+})$ be non negative and satisfy
\begin{equation*}
    \begin{cases}
    L_a v = 0 & \text{in } \R^{N+1}_+ \\
    \partial_{\nu}^a v \leq - \lambda v^p & \text{on } \R^{N}.
    \end{cases}
\end{equation*}
If the H\"older quotient of exponent $\gamma$ of $v$ is uniformly bounded, for some $\gamma \in [0,2s)$, then $v$ is constant.
\end{lemma}
\begin{proof}
When $p \leq 1$ the lemma follows directly from Lemma \ref{lem: decay with perturbations_s}: indeed, by translating
and scaling,
\[
    v(x_0,0) \leq \sup_{\partial^0 B_{r/2}(x_0,0)} v \leq \frac{1}{\lambda^{1/p} r^{2s/p}}\sup_{\partial^+ B_r(x_0,0)} v \leq C \frac{1+r^{\gamma}}{r^{2s/p}}\to0\quad\text{ as }r\to\infty.
\]
When $p > 1$, we start by showing that $v$ has a bounded trace on $\R^N$. Let us assume, on the contrary, that $v(x,0)$ is not uniformly bounded from above: by the uniform control on the H\"older seminorm, there exists a sequence $\{x_n\} \subset \R^N$ such that
\[
    M_n := \inf_{\partial^0B^+_1(x_n,0)} v^{p-1} \to +\infty.
\]
But then, restricting on $B^+_1(x_n,0)$, we have that $v \geq 0$ satisfies
\[
    \begin{cases}
    L_a v = 0 & \text{in } B^+_1(x_n,0) \\
    \partial_{\nu}^a v \leq -M_n v  & \text{on } \partial B^+_1(x_n,0)
    \end{cases}
\]
and, thanks to Lemma \ref{lem: decay with perturbations_s} (with exponent $1$ instead of $p$) and the H\"older continuity, we obtain
\[
    \begin{split}
    \inf_{\partial^0B^+_{1}(x_n,0)} v \leq \sup_{\partial^0B^+_{1/2}(x_n,0)} v \leq \frac{1}{M_n} \sup_{\partial^+B^+_1(x_n,0)} v 
    &\leq \frac{1}{M_n} \left(\inf_{\partial^0B^+_1(x_n,0)} v + C\right),
    \end{split}
\]
a contradiction. Let now $\{x_n\} \subset \R^N$ be a maximizing sequence of $v(x,0)$, that is
\[
    \sup_{x \in \R^N} v(x,0) = \lim_{n \to \infty} v(x_n,0) < \infty,
\]
and let us also introduce the sequence of functions
\[
    v_n(x,y) := v(x-x_n,y).
\]
The functions $v_n$ share the same uniform bound in $\C^{0,\gamma}$, so that we can pass to the uniform limit and find a limiting function $\bar v \in \C^{0,\gamma}(\overline{\R^{N+1}_+})$ which satisfies the assumptions of the lemma, its trace on $\R^N$ achieving the global maximum at $(0,0)$. Let us denote with $w$ the unique bounded $L_a$-harmonic extension of $\bar v(x,0)$ (which is defined since $\bar v(x,0)$ is bounded). We see that the odd extension across $\{y=0\}$ of the difference $w - \bar v$ satisfies the assumptions of Lemma \ref{lem: first collection_s}, yielding $\bar v \equiv w$. From the equation we deduce that
\[
    \partial_{\nu}^a \bar v(0,0) = - \lambda \bar v(0,0)^p = - \lambda \sup_{x \in \R^{N}} v(x,0)^p \leq 0
\]
and the Hopf Lemma implies $\bar v(0,0) = 0$, that is $v \equiv 0$.
\end{proof}
\section{The blow-up argument}

As we mentioned in the introduction, the proof of the a priori uniform $\C^{0,\alpha}$-bounds of solutions to problem $\problem{\beta}$ is based on a blow-up argument. To perform this technique, we will assume that the solutions are not a priori bounded in a uniform way in some H\"older norms and then, through a series of lemmas, we will show that this implies the existence of entire solutions to some limiting problem. The scheme of the proof here presented may resemble the one contained for instance in \cite{ctv} and also \cite{tvz1, tvz2}. However, in the present situation, some of the steps, which were adopted in the aforementioned papers, actually fail. This phenomenon is consequence of deep differences in the interaction between competition and diffusion features of the models. Once the blow-up procedure is completed, we will reach different contradictions in the next section, depending on the particular choice of $k$, $p$ and $a_{ij}$: for the moment, in what follows we will always assume that $p>0$, $a_{ij}>0$ for any $j\neq i$,
and that the reaction terms $f_{i,\beta}$ are continuous and map bounded sets into bounded sets, uniformly w.r.t. $\beta>0$ (notice that these are the common assumptions for all the statements in the introduction).

Let $\{\bv_\beta\}_\beta =\{(v_{1,\beta}, \dots, v_{k,\beta})\}_{\beta}$ denote a family of positive solutions
to problem $\problem{\beta}$, uniformly bounded in $B^+_1$.
We begin the analysis by recalling the regularity result which holds whenever $\beta$ is finite. For easier notation, we write $B^+=B^+_{1}$.
\begin{lemma}\label{lem: sire_dupaigne_s}
For every $0<\alpha<\min(2s,1)$, $\bar m>0$ and $\bar\beta>0$,
there exists a constant $C=C(\alpha,\bar m,\bar\beta)$ such that
\[
    \| \bv_\beta\|_{\C^{0,\alpha}\left(\overline{B^+_{1/2}}\right)} \leq C,
\]
for every $\bv_\beta$ solution of problem $\problem{\beta}$ on $B^+$, satisfying
\[
\beta\leq\bar\beta\quad\text{ and }\quad  \| \bv_{\beta} \|_{L^{\infty}(B^+)} \leq \bar m.
\]
\end{lemma}
\begin{proof}[Sketch of the proof]
Since the functions involved are a priori in $L^\infty(B^+)$, we can apply the regularity result in \cite[Lemma 4.1]{tvz2} to obtain regularity of the solutions in $\C^{0,\alpha}$ spaces for every $\alpha < \min(2s,1)$ (see also the proof of \cite[Lemma 2.3]{dipierro_Symmetry}).
\end{proof}
%

Let the cut-off function $\eta$ be smooth, with
\begin{equation*}
\eta(X) =
 \begin{cases}
 1   &  X\in B_{1/2}\\
 0   &  X\in\R^{N+1}\setminus B_1
 \end{cases} \quad \text{while} \quad \eta(X) \in (0,1)\text{ elsewhere.}
\end{equation*}
The rest of this section is devoted to the proof of the following proposition.
\begin{proposition}\label{prp: blow_up}
If there exists $0 < \alpha < \min(2s,1)$ such that
\[
    \sup_{\beta > 0} | \eta\bv_{\beta}|_{\C^{0,\alpha}\left(\overline{B^+}\right)} = + \infty
\]
then for a suitable choice of $\{r_\beta\}_\beta\subset\R^+$ and $\{x'_\beta\}_\beta \subset \R^N$, the blow-up family
\[
    w_{i,\beta}(X) := \eta(x'_\beta,0) \frac{v_{i,\beta}((x'_\beta,0) + r_\beta X)}{ r_\beta^{\alpha} | \eta\bv_{\beta}|_{\C^{0,\alpha}\left(\overline{B^+}\right)}}
\]
admits a convergent subsequence in the local uniform topology. Moreover the limit $\bw \in (H^{1;a}_{\loc}\cap \C^{0,\alpha})\left(\overline{\R^{N+1}_+}\right)$ enjoys the following properties:
\begin{enumerate}
    \item each $w_i$ is a $L_a$-harmonic function of $\R^{N+1}_+$;
    \item at least one component of $\bw$ is non constant, and it attains its maximal H\"older quotient of exponent $\alpha$ at a pair of points in the half-ball $\overline{B_1^+ }$;
    \item either there exists $M > 0$ such that
    \[
        \partial^a_\nu w_i = - M w_i^p \sum_{j \neq i} a_{ij} w_j^p \quad \text{on $\R^{N}$}
    \]
    or $w_i w_j |_{y = 0} = 0$ for every $j \neq i$ and
    \[
        \partial_{\nu}^a w_i \leq 0, \qquad  w_i \partial_{\nu}^a \left(w_i - \sum_{j\neq i} \frac{a_{ij}}{a_{ji}} w_j\right) = 0 \quad \text{on $\R^{N}$}.
    \]
\end{enumerate}
\end{proposition}
The proof is divided in several steps. First, we choose any subsequence $\bv_n:=\bv_{\beta_n}$ such that
\[
    \sup_{n \in \N} | \eta\bv_{n}|_{\C^{0,\alpha}\left(\overline{B^+}\right)} =: \sup_{n \in \N}  L_n  = + \infty,
\]
where by Lemma \ref{lem: sire_dupaigne_s} both $\beta_n\to\infty$ and the H\"older quotients $L_n$ are achieved,
say
\[
\begin{split}
    L_n := &\max_{i = 1, \dots, k} \max_{X'\neq X'' \in \overline{B^+}} \frac{|(\eta v_{i,n})(X')-(\eta v_{i,n})(X'')|}{|X'-X''|^{\alpha}}\\
    = &\frac{|(\eta v_{1,n})(X'_n)-(\eta v_{1,n})(X''_n)|}{r_n^{\alpha}},
\end{split}
\]
where we have written $r_n := |X'_n-X''_n|$. Finally, we are in a position to define the two blow-up sequences we will work with as
\[
    w_{i,n}(X) := \eta(x'_n,0) \frac{v_{i,n}({(x'_n,0)} + r_n X)}{L_n r_n^{\alpha}} \quad \text{and} \quad \bar{w}_{i,n}(X) := \frac{(\eta v_{i,n})({(x'_n,0)} + r_n X)}{L_n r_n^{\alpha}},
\]
both defined on the domain
\[
\tau_n B^+ := \frac{B^+ - {(x'_n,0)}}{r_n}.
\]
Accordingly, the corresponding reaction terms can be expressed as
    \[
        f_{i,n}(x,t_1, \dots, t_k) = r_n^{2s}  \frac{\eta({x'_n,0})}{ L_n r_n^{\alpha}} f_{i,\beta_n}\left(X'_n + r_n x,  t_1 \frac{L_n r_n^{\alpha}}{\eta({x'_n,0})}, \dots,  t_k \frac{L_n r_n^{\alpha}}{\eta({x'_n,0})}\right).
    \]
In \cite[Section 6]{tvz1} and \cite[Section 4]{tvz2} we have analyzed in detail the behavior of the two blow-up sequences in the different case of variational competition. In the following lemma we collect the initial remarks about such sequences, the proof of which is independent of the type of competition. In particular, we have that the domains exhaust the whole $\R^{N+1}_+$, and that the two sequences $\{\bw_{n}\}_n$ and $\{\bar \bw_{n}\}_n$ -- of which the former satisfies an equation and the latter has uniformly bounded H\"older quotient -- are close on any compact.
\begin{lemma}\label{lem: mega lemma_s}
As $n \to \infty$ the following assertions hold
\begin{enumerate}
    \item $r_n \to 0$, $\|f_{i,n}\|_\infty\to0$,  $\tau_n B^+ \to \R^{N+1}_+$ and $\tau_n \partial^0 B^+ \to \R^{N} \times \{0\}$;
    \item the sequence $\{\bw_{n}\}_n$ satisfies
\begin{equation}\label{eqn: w_sol_s}
    \begin{cases}
    L_a w_{i,n} = 0 & \text{ in }\tau_nB^{+}\\
    \partial_{\nu}^{a} w_{i,n} = f_{i,n}(x,w_{1,n}, \dots, w_{k,n}) - M_n w_{i,n}^p \tsum_{j \neq i}a_{ij} w_{j,n}^p & \text{ on } \tau_n \partial^0 B^+,
    \end{cases}
\end{equation}
    where
    \[
        M_n = \beta_n r_n^{2s} \left(\frac{\eta({x'_n,0})}{L_n r_n^{\alpha}}\right)^{1-2p};
    \]
    \item\label{itm: equihold} the sequence $\{\bar \bw_n\}_n$ has uniformly bounded $\C^{0,\alpha}$-seminorm, the oscillation of the first component in $B_1^+$ being always $1$;
    \item\label{itm: close_seqs} for any compact $K\subset\R^{N+1}$,
    \[
    \max_{X \in K\cap\overline{\tau_n B^+}} | \bw_{n}(X)- \bar{\bw}_{n}(X)| \to 0
    \]
    (and therefore also $\bw_{n}$ has uniformly bounded oscillation on $K$).
\end{enumerate}
\end{lemma}
In the next series of lemmas we are going to show that both sequences converge to the same blow-up limit.
To this end, we have to exclude the case in which the sequences are unbounded at the origin: indeed, the uniform boundedness of a sequence at some point is enough,
together with points \eqref{itm: equihold} and \eqref{itm: close_seqs} of the previous lemma, to conclude the convergence (uniform on compact sets) of the two sequences.

\begin{lemma}\label{lem: trip estimate}
For any $r > 0$ there exists a constant $C$ such that the estimate
\[
    M_n \int\limits_{\partial^0 B_r^+} \tsum_{j \neq i}a_{ij} w_{i,n}^{p+1} w_{j,n}^p \, \de{x} \leq C(r) (|w_{i,n}(0)| +1)
\]
holds uniformly in $n$.
\end{lemma}
\begin{proof}
Let us consider the quantities
\[
    \begin{split}
        E(r) &:= \frac{1}{r^{N+a-1} } \left( \int\limits_{B_r^+} y^{a}|\nabla w_{i,n}|^2 + \int\limits_{\partial^0 B_r^+}\left( -f_{i,n}w_{i,n} + M_n w_{i,n}^{p+1} \tsum_{j \neq i}a_{ij} w_{j,n}^p\right) \right)\\
        H(r) &:= \frac{1}{r^{N+a}} \int\limits_{\partial^+ B_r^+} y^a w_{i,n}^2,
    \end{split}
\]
where $H \in AC(R,2R)$, for any $R>0$ fixed and $n$ sufficiently large. If we test equation
\eqref{eqn: w_sol_s} by $w_{i,n}$ itself in the ball $B_r^+$, we obtain
\[
H'(r) = \frac{2}{r^{N+a}} \int\limits_{\partial^+ B_r^+} y^a w_{i,n} \partial_{\nu} w_{i,n}= \frac{2}{r} E(r),
\]
which can be integrated to infer
\begin{equation*}
    H(2R) - H(R) = \int\limits_{R}^{2R} \frac{2}{r} E(r)\,\de r.
\end{equation*}
On the one hand, the left hand side of of the previous identity can be estimated by recalling that $w_{i,n}$ has uniformly bounded oscillation on any compact set (Lemma \ref{lem: mega lemma_s},
\eqref{itm: close_seqs}):
\[
    \begin{split}
    H(2R) - H(R) &= \int\limits_{\partial^+ B^+}y^a \left[ w_{i,n}^2(2RX) - w_{i,n}^2(RX) \right] \de{\sigma} \\
    &= \int\limits_{\partial^+ B^+} y^a\left. w_{i,n} \right|^{2RX}_{RX} \left[ \left. w_{i,n} \right|^{2RX}_{0} + \left. w_{i,n} \right|^{RX}_{0} + 2w_{i,n}(0) \right] \de{\sigma}\\
    &\leq C(R)(|w_{i,n}(0)| +1).
    \end{split}
\]
On the other hand, we obtain a lower bound of the right hand side as
\[\begin{split}
    \int\limits_{r}^{2r} \frac{2}{s} E(s)  \mathrm{d}s &\geq \min_{s \in [r,2r]} E(s) \\
    &\geq \frac{1}{r^{N+a-1}} \left( \frac{M_n}{2^{N+a}} \int\limits_{\partial^0 B_r^+} \tsum_{j \neq i}a_{ij} w_{i,n}^{p+1} w_{j,n}^p \, \de{x} - \int\limits_{\partial^0 B_{2r}^+}  |f_{i,n}| w_{i,n} \, \de{x}  \right)\\
    &\geq C \left(M_n \int\limits_{\partial^0 B_r^+} \tsum_{j \neq i}a_{ij} w_{i,n}^{p+1} w_{j,n}^p \, \de{x} - \|f_{j,n}\|_{L^{\infty}} (|w_{i,n}(0)| +1)\right).\qedhere
\end{split}
\]
\end{proof}

\begin{lemma}\label{lem: bound in Mn not infinitesimal local_s}
If $\bar{w}_{i,n}(0) \to \infty$ for some $i$, then there exists $C$ such that
\[
    M_n \bar{w}_{i,n}^p(0) \leq C
\]
for a constant $C$ independent of $n$. In particular, $M_n\to0$.
\end{lemma}
\begin{proof}
Reasoning by contradiction we assume that $M_n \bar {w}_{i,n}^p(0) \to \infty$,
at least for a subsequence. For any $r > 0$ fixed, Lemma \ref{lem: mega lemma_s} forces
\[
    I_{r,n} := \inf_{\partial^0B_r^+} M_n w_{i,n}^p \to \infty.
\]
From Lemma \ref{lem: trip estimate}, we directly obtain
\begin{equation}\label{eqn: trip estimate}
    M_n \inf_{\partial^0 B^+_r}w_{i,n}^{p+1} \int_{\partial^0B_r^+}\tsum_{j \neq i}a_{ij} w_{j,n}^p \de x \leq C(r) (|w_{i,n}(0)| +1)
\end{equation}
that is, since $w_{i,n}(0)/w_{i,n}(x) \to 1$ uniformly in compact sets,
\[
    I_{r,n} \int_{\partial^0B_r^+}\tsum_{j \neq i}a_{ij} w_{j,n}^p \de x  \leq C.
\]
Let $j\neq i$. Since $I_{r,n} \to \infty$, we deduce that  $w_{j,n}\to0$ in $L^p(\partial^0B_r^+)$, for every $r$. Therefore Lemma \ref{lem: mega lemma_s} implies that both $\{\bar w_{j,n}\}_n$ and $\{w_{j,n}\}_n$ converge, uniformly on compact sets, to an $L_a$-harmonic function $w_{j,\infty} \in \C^{0,\alpha}(\R^{n+1}_+)$ such that
\[
    w_{j,\infty}(x,0) = 0\qquad\text{on }\R^N.
\]
The Liouville result in Lemma \ref{lem: first collection_s} applies to the odd extension of $w_{j,\infty}$
across $\{y=0\}$, yielding
\[
    w_{j,\infty} \equiv 0 \qquad \text{for $j \neq i$}.
\]
In particular, by uniform convergence, the unitary H\"older quotient is not achieved by any of the functions $\bar w_{j,n}$ for $j \neq i$ and $n$ large enough: it follows that we must have $i = 1$.

Now, let us recall that each $w_{j,n}$ with $j \neq 1$ satisfies the inequality
\begin{equation}\label{eqn: inequality for decay}
    \begin{cases}
    L_a w_{j,n} = 0 &\text{in $B_{2r}^+$}\\
    \partial_{\nu}^a w_{j,n} \leq \|f_{j,n}\|_{L^{\infty}(B_{2r})} - a_{ji}I_{2r,n} w_{j,n}^p &\text{in $\partial^0 B_{2r}^+$}
    \end{cases}
\end{equation}
so that by Lemma \ref{lem: decay with perturbations_s} we have the estimate
\[
    \sup_{\partial^0 B_r^+} w_{j,n}^p \leq \frac{C(r)}{ I_{2r,n}} \sup_{\partial^+B^+_{2r}} w_{j,n}^p.
\]
On the other hand, the function $w_{1,n}$ satisfies a boundary condition that can be estimated as
\[
\begin{split}
    \sup_{\partial^0 B_{r}^+} |\partial_{\nu}^a w_{1,n}| &\leq \|f_{1,n}\|_{L^{\infty}(B_{2r})} + I_{r,n} \tsum_{i \neq 1} a_{ij}\sup_{\partial^0 B_{r}^+} w_{j,n}^p \\
    &\leq \|f_{1,n}\|_{L^{\infty}(B_{2r})} + C(r) \frac{I_{r,n}}{I_{2r,n}} \tsum_{i \neq 1} \sup_{\partial^+B^+_{2r}} w_{j,n}^p \to 0,
\end{split}
\]
where we used the fact that
\[
\inf_{\partial^0B_{2r}^+} \bar w_{i,n} \leq \inf_{\partial^0B_{r}^+} \bar w_{i,n}
\leq \inf_{\partial^0B_{2r}^+} \bar w_{i,n} + Cr^\alpha
\quad \implies \quad
\lim_{n \to \infty} \frac{I_{r,n}}{I_{2r,n}} =  1.
\]
Let us now introduce the sequences
\[
    W_{1, n}(x,y) := w_{1,n}(x,y) - w_{1,n}(0,0), \qquad \bar W_{1, n}(x,y) := \bar w_{1,n}(x,y) - \bar w_{1,n}(0,0).
\]
As before, we can use Lemma \ref{lem: mega lemma_s} to prove that both sequences converge to the same $L_a$-harmonic
function, which is globally H\"older continuous, non constant, and which has trivial conormal derivative on $\R^N$,
in contradiction with Lemma \ref{lem: prescribed constant normal derivative_s}.
\end{proof}
\begin{lemma}\label{lem: w(0) bdd_s}
The sequence  $\{\bar{\bw}_{n}(0)\}_{n \in \N}$ is bounded.
\end{lemma}
\begin{proof}
By contradiction, let $\{\bar{\bw}_{n}(0)\}_{n \in \N}$ be unbounded. Then, by the previous lemma, $M_n\to0$.
To start with, we claim that for every $j$ there exists a constant $\lambda_j \geq 0$ such that, up to subsequences,
\[
    M_n \bar w_{j,n}^p \to \lambda_j \qquad \text{locally uniformly}.
\]
Indeed, if $\bar w_{j,n}(0)$ is bounded this follows by uniform H\"older bounds, with $\lambda_j=0$; if it is unbounded, from Lemma \ref{lem: bound in Mn not infinitesimal local_s} we obtain that $M_n \bar w_{j,n}^p(0) \to \lambda_j$, while
\[
    \sup_{\partial^0 B^+_r} |M_n \bar w_{j,n}^p- M_n \bar w_{j,n}^p(0) | = M_n \bar w_{j,n}^p(0) \sup_{\partial^0 B^+_r} \left|\left(\frac{\bar w_{j,n}}{\bar w_{j,n}(0)}\right)^p -1\right| \to 0.
\]

Now, let $i$ be such that $w_{i,n}(0)$ is bounded. As usual, we can use Lemma \ref{lem: mega lemma_s}
to show that $w_{i,n} \to w_{i,\infty} \in \C^{0,\alpha}(\R^{N+1}_+)$ in the local uniform topology, where, using the claim above, $w_{i,\infty}$ is a solution to
\[
    \begin{cases}
    L_a w_{i,\infty}= 0 &\text{ in } \R^{N+1}_+ \\
    \partial_{\nu}^a w_{i,\infty} = - w_{i,\infty}^p \sum_{j} a_{ij}\lambda_j  &\text{ on } \R^N.
    \end{cases}
\]
Lemma \ref{lem: global eigenfunction_s} then implies $w_{i,\infty} \equiv 0$: in particular, we have that $\bar w_{1,n}(0)$ is unbounded.

Let us then turn our attention to $w_{1,n}$. Again, if $j$ is such that $\bar w_{j,n}(0)$ is bounded, then by the previous discussion $\bar w_{j,n} \to 0$ locally uniformly and
\[
    \underbrace{M_n \bar{w}_{1,n}^p}_{\leq C \text{ (Lemma \ref{lem: bound in Mn not infinitesimal local_s})}} \bar{w}_{j,n}^p  \to 0.
\]
Otherwise, if $j$ is such that $\bar w_{j,n}(0)$ is unbounded, then Lemma \ref{lem: trip estimate} provides
\[
    C\geq M_n {w}_{1,n}(0)^p w_{j,n}(0)^p  \int\limits_{\partial^0 B_r^+} \tsum_{j \neq 1}a_{1j} \frac{w_{1,n}^{p+1}}{w_{1,n}(0)^p  (|w_{1,n}(0)| +1) } \frac{w_{j,n}^p}{w_{j,n}(0)^p} \, \de{x}
\]
so that $M_n {w}_{1,n}(0)^pw_{j,n}(0)^p$ is uniformly bounded. Since if $\{{w}_{j,n}(0)\}_{n \in \N}$ is unbounded then also $\{{w}_{j,n}(x)\}_{n \in \N}$ is, for any fixed $x$, and the same argument shows that $M_n {w}_{1,n}(x)^p w_{j,n}(x)^p$ is bounded. Now,
\begin{multline*}
    M_n \left|{w}_{1,n}(x)^p w_{j,n}(x)^p - {w}_{1,n}(0)^p w_{j,n}(0)^p\right| \\ \leq
    M_n {w}_{1,n}(x)^p w_{j,n}(x)^p \left| 1 - \frac{w_{j,n}(0)^p}{w_{j,n}(x)^p}\right| +
    M_n {w}_{1,n}(0)^p w_{j,n}(0)^p \left| \frac{w_{1,n}(x)^p}{w_{1,n}(0)^p}-1\right| \to0.
\end{multline*}
This shows the existence of a constant $\lambda$ such that, at least up to a subsequence,
\[
    f_{1,n}  -M_n \bar w_{1,n}^p \tsum_{j \neq 1} a_{1j}\bar w_{j,n}^p \to \lambda
\]
uniformly on every compact subset of $\R^N$, and the same holds true for the sequence $\{w_{1,n}\}_{n \in \N}$. Thus, as usual, $W_{1,n} = w_{1,n}-w_{1,n}(0)$ converges to $W_1$ which is nonconstant, globally H\"older continuous of exponent $\alpha < \min(1,2s)$, and which solves
\[
    \begin{cases}
    L_a W_1= 0 &\text{ in } \R^{N+1}_+ \\
    \partial_{\nu}^a W_1 = \lambda &\text{ on } \R^N.
    \end{cases}
\]
Invoking Lemma \ref{lem: prescribed constant normal derivative_s}, we obtain a contradiction.
\end{proof}

The boundedness of the sequences $\{\bar{\bw}_{n}(0)\}_{n \in \N}$ implies, by Lemma \ref{lem: mega lemma_s}, the convergence of both $\{\bar{\bw}_{n}\}_{n \in \N}$ and $\{\bw_{n}\}_{n \in \N}$ to the same blow-up limit.  Reasoning as in the proof of \cite[Lemma 6.13]{tvz1}, one can show that the convergence is also strong in the natural Sobolev space.
\begin{lemma}\label{lem: uniform implies strong convergence local_s}
There exists $\bw \in (H^{1;a}_{\loc}\cap \C^{0,\alpha})\left(\overline{\R^{N+1}_+}\right)$ such that, up to a subsequence,
\[
	\bw_{n} \to \bw \text{ in }(H^{1;a}\cap C)(K)
\]
for every compact $K\subset\overline{\R^{N+1}_+}$. Furthermore, each $w_i$ is $L_a$-harmonic, and $w_1$ is non constant.
\end{lemma}

Depending on the behavior of the sequence $M_n$, the limiting functions $\bw$ satisfy a different limiting problem: first, we can exclude the case $M_n \to 0$.
\begin{lemma}
There exists $C>0$ such that $M_n \geq C$.
\end{lemma}
\begin{proof}
Let us assume that there exists a subsequence $M_{n_k}$ that converges to $0$. Passing to the limit in the sequence, we obtain as a limiting problem
\[
    \begin{cases}
        L_a w_i = 0 &\text{in $\R^{N+1}_+$}\\
        \partial_{\nu}^a w_i = 0 &\text{on $\R^{N}$}.
    \end{cases}
\]
By Lemma \ref{lem: first collection_s} each $w_i$ is constant, and this is in contradiction with the fact that $w_1$
oscillates in the half-ball $B^+$.
\end{proof}
\begin{lemma}\label{lem: Mn to 1}
If $M_n \to M>0$, then the blow-up profiles $\bw$ solve
\[
    \begin{cases}
        L_a w_i = 0 &\text{in $\R^{N+1}_+$}\\
        \partial_{\nu}^a w_i = -M w_i^p \sum_{j \neq i} a_{ij} w_j^p &\text{on $\R^{N}$}.
    \end{cases}
\]
\end{lemma}
\begin{proof}
This is a direct consequence of Lemma \ref{lem: uniform implies strong convergence local_s}.
\end{proof}
To conclude the proof of Proposition \ref{prp: blow_up} we are left to analyze the case $M_n \to \infty$.
\begin{lemma}\label{lem: Mn to infty}
If $M_n \to \infty$, then the blow-up profiles $\bw$ are such that $w_i w_j|_{y=0} = 0$, for every $j \neq i$, and
\begin{equation}\label{eqn: capucci}
    \begin{cases}
        \partial_{\nu}^a w_i \leq 0 \\
        \partial_{\nu}^a \left(w_i - \tsum_{j\neq i} \frac{a_{ij}}{a_{ji}} w_j\right) \geq 0 \\
        w_i \partial_{\nu}^a \left(w_i - \tsum_{j\neq i} \frac{a_{ij}}{a_{ji}} w_j\right) = 0, \\
    \end{cases}
\end{equation}
where the inequalities are understood in the sense of $\R^N$-measures.
\end{lemma}
%
%
\begin{proof}
For any nonnegative $\psi\in \C^{\infty}_0(\R^{N+1})$, we test equation \eqref{eqn: w_sol_s} to find
\[
    0 \leq \int\limits_{\partial^0 B_r} M_n w_{i,n}^p \tsum_{j \neq i} a_{ij} w_{j,n}^p \psi \leq \int\limits_{\partial^0 B_r} \left(f_{i,n} \psi - w_{i,n} \partial^{a}_{\nu}\psi \right) - \int\limits_{B_r^+} w_{i,n} L_a \psi.
\]
Since the right hand side is bounded by local uniform convergence, we infer that
\begin{equation}\label{eqn: bounded competition}
    M_n \int_{K} w_{i,n}^p w_{j,n}^p \, \de x \leq C(K) \qquad \forall j \neq i,
\end{equation}
for any compact set $K \subset \R^N$. In particular it follows that, at the limit, $w_i w_j |_{y = 0} = 0$ for every $j \neq i$. Furthermore, the first and the second inequalities in \eqref{eqn: capucci} follow from equation
\eqref{eqn: w_sol_s} and from the fact that, for every $n$,
\[
    \partial_{\nu}^a \left(w_{i,n} - \tsum_{j\neq i} \frac{a_{ij}}{a_{ji}} w_{j,n}\right) = f_{i,n} - \tsum_{j\neq i} \frac{a_{ij}}{a_{ji}} f_{j,n} + M_n \tsum_{\substack{j\neq i\\ h \neq i,j}} \frac{a_{ij} }{ a_{ji}} a_{jh} w_{j,n}^p w_{h,n}^p
\]
(we recall that the reaction terms $f_{i,n} \to 0$ uniformly in $\R^N$). Finally, the identity in
\eqref{eqn: capucci} can be obtained by multiplying the previous equation by $w_{i,n}$, once one can estimate the
terms $M_n w_{i,n} w_{j,n}^p w_{h,n}^p$. To this aim, let $\eps > 0$,
and let us define the (possibly empty) set
\[
    \supp_{i}^\eps = \{x \in \R^N: w_{i}(x,0) \geq \eps \}.
\]
We observe that for any $K \subset \R^N$ compact set, the local uniform convergence of the sequence $\{\bw_n\}$ implies
\[
    \begin{cases}
        w_{i,n}(x,0) \geq \frac{\eps}{2} &\forall x \in K \cap \supp_{i}^{\eps}\\
        w_{i,n}(x,0) \leq 2\eps &\forall x \in K \setminus \supp_{i}^{\eps},
    \end{cases}
\]
for any $n$ large enough. As a consequence
\begin{multline*}
    M_n \int_{K} w_{i,n} w_{j,n}^p w_{h,n}^p\mathrm{d}x \leq M_n \int\limits_{K \setminus \supp_{i}^\eps} w_{i,n} w_{j,n}^p w_{h,n}^p \mathrm{d}x + M_n \int\limits_{K \cap \supp_{i}^\eps} w_{i,n} w_{j,n}^p w_{h,n}^p \mathrm{d}x \\
        \leq M_n 2\eps \int\limits_{K \setminus \supp_{i}^\eps} w_{j,n}^p w_{h,n}^p \mathrm{d}x + M_n \int\limits_{K \cap \supp_{i}^\eps} w_{i,n} 2^{2p} \frac{\left(1+\|f_{j,n}\|\right)^p}{M_n \eps^p} \frac{\left(1+\|f_{h,n}\|\right)^p}{M_n \eps^p} \mathrm{d}x \\
        \leq C \left( \eps + \frac{1}{M_n} \eps^{-2p}\right),
\end{multline*}
where we used estimate \eqref{eqn: bounded competition} and Lemma \ref{lem: decay with perturbations_s} to bound the two terms. Choosing $n$ sufficiently large so that $\eps^{-2p} \leq \eps M_n$, we conclude by the arbitrariness of $\eps$ that
\[
    \lim_{n \to \infty} M_n \int_{K} w_{i,n} w_{j,n}^p w_{h,n}^p\mathrm{d}x = 0 \qquad \text{for every $i \neq j \neq h$.}\qedhere
\]
\end{proof}
\begin{corollary}\label{cor: blow-up_aij=1}
Let $a_{ij}=1$ for every $i,j$ and $\bw$ be a blow-up profile. For every $i\neq j$ the functions $z= w_i-w_j$ are such that
\[
\begin{cases}
        L_a z^\pm \leq 0 &\text{in $\R^{N+1}_+$}\\
        z^\pm \partial_{\nu}^a z^\pm \leq 0 &\text{on $\R^{N}$}.
    \end{cases}
\]
\end{corollary}
\begin{proof}
A subtraction of the equation satisfied by $w_{i,n}$ and $w_{j,n}$ yields
\begin{multline*}
(w_{i,n} - w_{j,n})^\pm\partial_{\nu}^a (w_{i,n} - w_{j,n}) = (f_{i,n} -f_{j,n})(w_{i,n} - w_{j,n})^\pm \\
- M_n\underbrace{ (w_{i,n} - w_{j,n})^\pm(w_{i,n}^p-w_{j,n}^p)}_{\geq0} \sum_{h \neq i, j} w_{h,n}^p.\qedhere
\end{multline*}
\end{proof}

\section{Uniform H\"older bounds}

We are ready to show the almost optimal uniform H\"older bounds, in the case of two competing species. This will be
a consequence of the following lemma.
\begin{lemma}\label{lem: 3}
Under the assumption of Proposition \ref{prp: blow_up}, at least three components of the blow-up profile $\bw$
are non constant, and all the constant components are trivial.
\end{lemma}
\begin{proof}
We start by observing that each constant component has to be trivial: this is a direct consequence of the segregation condition $w_i w_j|_{y=0} = 0$ in the case $M_n \to \infty$, while, if $M_n \to M>0$, it is implied by the boundary condition and the fact that at least $w_1$ is non constant.

If $w_1$ is the only non constant component, then we obtain a contradiction with Lemma
\ref{lem: prescribed constant normal derivative_s} since in both cases $M_n \to M$
(Lemma \ref{lem: Mn to 1}) and $M_n \to \infty$ (Lemma \ref{lem: Mn to infty}), we have $\partial_\nu^a w_1 = 0$.

Let us now assume that only $w_1$ and, say, $w_2$ are non constant. Invoking again Lemma \ref{lem: Mn to 1} and \ref{lem: Mn to infty}, we obtain in both cases that
\[
    \begin{cases}
        L_a \left( a_{21}w_1-{a_{12}} w_2\right) = 0 &\text{in $\R^{N+1}_+$}\\
        \partial_{\nu}^a \left( {a_{21}}w_1-{a_{12}} w_2\right) = 0 &\text{on $\R^{N}$}.
    \end{cases}
\]
The application of Lemma \ref{lem: prescribed constant normal derivative_s} then implies
\[
 w_1 = C + \frac{a_{12}}{a_{21}} w_2,
\]
where, up to a permutation between $w_1$ and $w_2$, we may assume that the constant $C$ is non negative. If $M_n \to \infty$, the segregation condition $w_1 w_2 |_{y=0} = 0$ yields
\[
    \left(C + \frac{a_{12}}{a_{21}} w_2 \right) w_2 |_{y=0} = 0 \implies C = w_1 = w_2 = 0,
\]
a contradiction. In the remaining case, the function $w_2$ solves
\[
    \begin{cases}
        L_a w_2 = 0 &\text{in $\R^{N+1}_+$}\\
        \partial_{\nu}^a w_2 = - M w_2^p \left( C + \frac{a_{12}}{a_{21}} w_2\right)^p \leq -C'w_2^{2p} &\text{on $\R^{N}$,}
    \end{cases}
\]
in contradiction with Lemma \ref{lem: global eigenfunction_s}.
\end{proof}
\begin{proof}[Proof of Theorem \ref{thm: intro_hold2}]
The lemma above, combined with Proposition \ref{prp: blow_up}, provides all the results in the theorem but
the $H^{1;a}$ convergence; this last property follows from the uniform H\"older bounds, reasoning as in the proof
of Lemma \ref{lem: uniform implies strong convergence local_s}.
\end{proof}
Now we turn to the case of $k\geq3$ densities. We first prove uniform H\"older bounds with small exponent, when
the power $p$ is greater or equal than $1$. In order to
quantify such exponent, we need to introduce some notation. For any $\omega \subset \S_+^{N}:=\partial^+ B^+_1$ we consider the first eigenvalue of (the angular part of) $L_a$, defined as
\[
    \lambda_1(\omega) = \inf \left\{ \int\limits_{\S_+^{N}} |y|^a |\nabla_{T} u|^2 \, \de{\sigma}: u\equiv0\text{ on }\S_+^N\setminus\omega,\, \int\limits_{\S_+^{N}}  |y|^a u^2 \, \de{\sigma} = 1 \right\}
\]
(here $\nabla_T$ denotes the tangential part of the gradient), and the associated characteristic function
\[
    \gamma( t ) = \sqrt{ \left(\frac{N-2s}{2}\right)^2 + t } - \frac{N-2s}{2}.
\]
We are ready to state the following Liouville type result.
\begin{proposition}\label{prp: small_exp}
Under the assumption
\[
p\geq 1,
\]
let $\bw$ denote a blow-up limit as in Proposition \ref{prp: blow_up}, and let
\[
\nu=\nu(s,N):=\inf\left\{\frac{\gamma(\lambda_1(\omega_1))+\gamma(\lambda_1(\omega_2))}{2} :
\omega_i\subset\S_+^N,\, \omega_1\cap\omega_2\cap\{y=0\} = \emptyset \right\}.
\]
If
\[
|\bw(X)|\leq C(1+|X|^\alpha),\qquad\text{for some }\alpha<\nu,
\]
then $k-1$ components of $\bw$ are trivial.
\end{proposition}
\begin{remark}
As shown in \cite[Lemma 2.5]{tvz1}, \cite[Lemma 2.3]{tvz2}, $\nu(s,N)>0$ (and $\nu(s,N)\leq s$) for
every $0<s<1$, $N\geq1$.
\end{remark}
\begin{proof}
The proof is a byproduct of arguments already exploited in \cite{tvz2}. The first step consists in
obtaining a monotonicity formula of Alt-Caffarelli-Friedman type, with exponent between $\alpha$ and $\nu$.
In the case in which $\bw$ has segregated traces on $\{y=0\}$, this is \cite[Proposition 4]{tvz2}.
When $\bw$ satisfies a differential system, this can be done as in \cite[Proposition 5]{tvz2}, with minor changes:
namely, by replacing the term $v_i^2v_j^2$ with $v_i^{p+1}v_j^p$ (this can be done as far as $p\geq1$).

Once the validity of the monotonicity formula holds, one can deduce a related minimal growth rate for $\bw$, which is consistent with the one in the assumption only if all the components but one vanish.
\end{proof}
The result above can be improved, also removing the restriction on $p$, in the case of equal competition rates.
\begin{proposition}\label{prp: equal_aij}
Under the assumption
\[
a_{ij}=1\qquad\text{for every }1\leq i,j \leq k,
\]
let $\bw$ denote a blow-up limit as in Proposition \ref{prp: blow_up}, and let
\[
\mu=\mu(s,N):=\inf\left\{\frac{\gamma(\lambda_1(\omega_1))+\gamma(\lambda_1(\omega_2))}{2} :
\omega_i\subset\S_+^N,\, \omega_1\cap\omega_2 = \emptyset \right\}.
\]
If
\[
|\bw(X)|\leq C(1+|X|^\alpha),\qquad\text{for some }\alpha<\mu,
\]
then $k-1$ components of $\bw$ are trivial.
\end{proposition}
\begin{remark}\label{rem: mu1/2}
It is immediate to check that $\mu(s,N)\geq\nu(s,N)$ for every $s$, $N$. In particular, it is always 
positive. Were $\mu(s,N)=1$, we would find regularity up to $\alphaopt$ also in the case $k\geq3$. 
As a matter of fact, at the end of this section we will show that $1/2 \leq \mu(s,N) \leq 1$ for 
every $0<s<1$, $N\geq1$. Furthermore, it is proved in \cite{acf,csbook} that
\begin{equation}\label{eqn: sharp_k>2}
\mu\left(\frac12,N\right)=1, \qquad \text{for every }N\geq1.
\end{equation}
\end{remark}
\begin{proof}
We start showing that, for any choice $i\neq j$, if $w_i(\cdot,0) \leq w_j(\cdot, 0)$ then $w_i\equiv0$.
Indeed, if $\bw$ solves the differential system, then
\[
\partial_{\nu}^a w_i \leq - M w_i^p w_j^p \leq - M w_i^{2p},
\]
and the claim follows by Lemma \ref{lem: global eigenfunction_s}; in the case of segregated traces, then
$w_i(\cdot,0)\equiv 0$ and one can conclude by applying Lemma \ref{lem: first collection_s} (to the odd
extension of $w_i$ across $\{y=0\}$).

On the other hand, let us assume by contradiction that, for some $i\neq j$, the functions $z^\pm :=
(w_i - w_j)^\pm$ are both nontrivial. Then they satisfy the inequalities in Corollary \ref{cor: blow-up_aij=1},
and furthermore
\[
|z^\pm (X)|\leq C(1+|X|^\alpha),
\]
where $\alpha<\mu$. Under these assumptions, we can obtain a contradiction by reasoning as in the proof of
Proposition \ref{prp: small_exp}. To this aim, the only missing ingredient is the following
monotonicity formula.
\end{proof}
\begin{lemma}\label{thm: ACF_s}
Let $z_1, z_2 \in H^{1;a}(B_R^+(x_0,0))$ be continuous nonnegative functions such that
\begin{itemize}
   \item  $z_1 z_2 = 0$, $z_i(x_0,0) = 0$;
   \item  for every non negative $\phi \in \C_0^{\infty}(B_R(x_0,0))$,
   \[
   \int\limits_{\R^{N+1}_+}(L_a z_i)z_i\phi \, \de x \de y + \int\limits_{\R^N} (\partial_\nu^a z_i)z_i\phi \,
   \de x = \int\limits_{\R^{N+1}_+}y^a \nabla z_i\cdot\nabla(z_i\phi) \, \de x \de y \leq0.
   \]
\end{itemize}
Then
\[
    \Phi(r) := \prod_{i=1}^{2} \frac{1}{r^{2\mu}} \int\limits_{B_r^+(x_0,0)}
\frac{y^a|\nabla z_i|^2}{|X-(x_0,0)|^{N-2s} } \,\de{x} \de{y}
\]
is monotone non decreasing in $r$ for $r \in (0,R)$, where $\mu$ is defined as in Proposition \ref{prp: equal_aij}.
\end{lemma}
\begin{proof}
First of all we observe that, up to an even extension of the functions $z_i$ across $\{y=0\}$, the formula above is implied by the analogous one stated on the whole $B_r$. This latter formula, when $s=1/2$,
is nothing but the classical Alt-Caffarelli-Friedman one \cite{acf}. On the other hand, when $s\neq 1/2$,
its proof resemble the usual one, as done for instance in \cite{csbook} (see also \cite[Section 2]{tvz2} for further details).
\end{proof}
To conclude the proof of Theorem \ref{thm: intro_holdk}, we provide the following rough elementary estimate of $\mu(s,N)$ for $s\neq 1/2$.
\begin{lemma}\label{lem: elementary estimate}
For every $0<s<1$ and $N\geq1$ it holds
\[
\mu(s,N)\geq\frac12 .
\]
\end{lemma}
\begin{proof}
By trivial extension to higher dimensions of the eigenfunctions involved, it is easy to prove
that $\mu$ is decreasing with respect to $N$, thus we can assume $N\geq2$.

Let $\omega_1,\omega_2\subset\S^N_+$, $\omega_1\cap\omega_2=\emptyset$, and let $\phi_i\in H^{1;a}(\S^N_+)$
be the first eigenfunction associated to $\lambda_1(\omega_i)$ enjoying the normalization
\[
    \int_{\S^{N}} |y|^a \phi_i^2 \mathrm{d}\sigma = 1 \qquad \text{for $i =1, 2$}.
\]
If $\mathcal{R}$ denotes the Rayleigh quotient associated to $\lambda_1$, then we have that
\[
\begin{split}
\lambda_2(\S^N_+) &:= \inf_{\substack{V\subset H^{1;a}(\S^N_+) \\ \dim V\geq 2}}\max_{V} \mathcal{R}
 \leq \max_{\vartheta} \mathcal{R}\left(\phi_1 \cos\vartheta + \phi_2\sin \vartheta \right) \leq
 \max(\lambda_1(\omega_1),\lambda_1(\omega_2)).
\end{split}
\]
By monotonicity of $\gamma$ we  obtain that
\[
\mu(s,N)= \inf_{\omega_1\cap\omega_2=\emptyset} \frac{\gamma(\lambda_1(\omega_1))+\gamma(\lambda_1(\omega_2))}{2}
\geq \frac{1}{2}\gamma(\lambda_2(\S^N_+)).
\]
To conclude the proof, we show that $\gamma(\lambda_2(\S^N_+))=1$. Indeed, let $\psi_2$ be a second eigenfunction. Then its conormal derivative on $\partial \S^N_+$ is identically zero, and it can be extended in an even way across $\{y=0\}$
to an eigenfunction of $\S^N$. Moreover, by the well known properties of $\gamma$, we have that the function
\[
v(X)= |X|^{\gamma(\lambda_2(\S^N_+))} \psi_2\left(\frac{X}{|X|}\right)
\]
is $L_a$-harmonic up to $0$ (this is true, actually, because we are assuming $N\geq2$), is $y$-even, and has bounded growth. By Lemma \ref{lem: first collection_s} we deduce that, up to a rotation in the $x$ plane, $v=x_1$, concluding the proof.
\end{proof}
\begin{proof}[Proof of Theorem \ref{thm: intro_holdk}]
The uniform H\"older bounds with exponent $\alpha^*$ are obtained by combining Lemma \ref{lem: 3} with either Proposition \ref{prp: small_exp} (with $\alpha^*=\min(2s,\nu(s,N))$) or Proposition \ref{prp: equal_aij} (with $\alpha^*=\min(2s,\mu(s,N))$), respectively. In the second case, the exact value of $\alpha^*$ is provided by Remark \ref{rem: mu1/2} when $s=1/2$, and by Lemma \ref{lem: elementary estimate} when $s<1/4$.
\end{proof}
\begin{remark}
By comparison with the nodal partition of $\S^N_+$ associated to the homogeneous, $L_a$-harmonic function $v(x,y)=x_1$, we infer that
\[
\mu(s,N) \leq 1.
\]
\end{remark}

\section{Further properties of the segregation profiles}\label{sec: lim_prof}
In this last section we deal with the proof of Theorems \ref{thm: intro_lip2} and 
\ref{thm: intro_lastholdk}. Together
with the previous assumptions, in what follows we further suppose that the reaction terms 
$f_{i,\beta}\to f_i$ as $\beta\to\infty$, uniformly on compact sets, with $f_i$ Lipschitz 
continuous.

As a result of the previous sections, we have shown that $L^{\infty}$ uniform bounds on a family of 
solutions to the problem $\problem{\beta}$ is enough to ensure equicontinuity of the family 
independently from the competition parameter $\beta$. Reasoning as in the proof of Lemmas \ref{lem: 
uniform implies strong convergence local_s} and \ref{lem: Mn to infty} we deduce the following result.
\begin{proposition}\label{prp: segr_prop}
Any sequence  $\{\bv_{\beta_n}\}_{n \in \N}$, $\beta_n \to \infty$, of solutions to $\problem{\beta}$ which is uniformly bounded in $L^{\infty}(B^+)$ admits a subsequence which converges to a limiting profile
$\bv \in (H^{1;a}\cap \C^{0,\alpha})_{\loc}(B^+)$, for some $\alpha > 0$. Moreover
\begin{equation}\label{eqn: segr prop}
    \begin{cases}
        L_a v_i = 0 &\text{in $B^+$},\medskip\\
        \partial_\nu^a v_i \leq f_{i}(x, v_1, \dots, v_k)\smallskip\\
        \partial_\nu^a\left( v_i-\sum_{j \neq i} \frac{a_{ij}}{a_{ji}} v_j\right) \geq f_i-\sum_{j \neq i} \frac{a_{ij}}{a_{ji}} f_j &\text{on $\partial^0 B^+$},\smallskip\\
        v_i \cdot \left[ \partial_\nu^a\left( v_i-\sum_{j \neq i} \frac{a_{ij}}{a_{ji}} v_j\right) - f_i +\sum_{j \neq i} \frac{a_{ij}}{a_{ji}} f_j\right] = 0
    \end{cases}
\end{equation}
and $v_i(x,0)\cdot v_j(x,0) \equiv 0$ for every $j\neq i$.
\end{proposition}
%
After Proposition \ref{prp: segr_prop}, the optimal regularity for the case of two densities is almost
straightforward.
\begin{proof}[Proof of Theorem \ref{thm: intro_lip2}]
For a limiting profile $\bv=(v_1,v_2)$, let $w=a_{21}v_1 - a_{12}v_2$. Then Proposition \ref{prp: segr_prop}
implies that
\[
v_1(x,0)=\frac{1}{a_{21}}w^+,\qquad v_2(x,0)=\frac{1}{a_{12}}w^-,
\]
and
\[
    \begin{cases}
        L_a w = 0 &\text{in $B^+$},\\
        \partial_\nu^a w = g(w) &\text{on $\partial^0 B^+$},
    \end{cases}
\]
where
\[
    g(x,t) := a_{21} f_1\left(x, \frac{1}{a_{21}}t^+, \frac{1}{a_{12}}t^-\right) - 
    a_{12}f_2\left(x, \frac{1}{a_{21}}t^+, \frac{1}{a_{12}}t^-\right)
\]
is Lipschitz continuous. As a consequence, standard regularity (e.g. \cite[Proposition 2.8]{silve}, \cite[Lemma 2.1]{dipierro_Symmetry}) applies, providing that $w\in \C^{1,\alpha}$ and thus $u_1$, $u_2$ are Lipschitz continuous.
\end{proof}
\begin{remark}
An important consequence of the argument above is that whenever there are only two species that are segregated, under
suitable growth conditions about $f_1$, $f_2$ the corresponding free boundary
\[
    \Gamma := \{x \in \partial^0B: v_1(x,0) = v_2(x,0) = 0\}
\]
is a closed set of empty interior (in the $N$ dimensional topology). Indeed $w=a_{21}v_1 - a_{12}v_2$ satisfies a semilinear equation for which unique continuation holds, see \cite[Theorems 1.4, 4.1]{fafe}.
\end{remark}

We are left to deal with the case $k\geq3$ for the half-laplacian, i.e.
\[
s=\frac12.
\]
In this case, by Theorem \ref{thm: intro_holdk} we already know that, when $a_{ij}=1$, the traces 
of the limiting profiles enjoy almost Lipschitz continuity on $K\cap\{y\geq0\}$, for every compact 
$K\subset B$. We are going to show 
that the same holds also for general $a_{ij}$, when there are no internal reaction terms in a 
neighborhood of the free boundary. More precisely, we assume
that the Lipschitz continuous functions $f_i$ are such that
\[
f_i(x,t_1, \dots, t_k) \equiv 0\quad\text{ whenever }|(t_1,\dots, t_k)| < \theta,
\]
for some $\theta > 0$ (such assumption can be weakened, but we prefer to avoid further 
technicalities at this point). Finally, $K\subset B$ will denote a fixed compact set.
\begin{remark}\label{rem: red to bdd}
As before, since the components of a limiting profile $\bv$ are harmonic on $B^+$, we have that its regularity on $K$ is directly connected to the regularity of the same function in $K\cap\{0\leq y<\eps\}$ for arbitrarily small $\eps>0$.
\end{remark}
\begin{definition}\label{def: hat}
For any function $\bv \in H^1 \cap \C(B^+; \R^k)$ which satisfies \eqref{eqn: segr prop} (with $s = 1/2$), we let $\hat \bv := (\hat v_1, \dots, \hat v_k)$ where
\[
    \hat v_i(x,y) = v_i(x,y) - \sum_{j \neq i} \frac{a_{ij}}{a_{ji}} v_j(x,y).
\]
\end{definition}
To clarify the effect of the segregation condition, we introduce the definition of multiplicity of boundary points.
\begin{definition}
We define the multiplicity of a point $x \in \partial^0 B^+$  as
\[
    m(x) := \sharp \left\{i : \mathcal{H}^N(\{v_i(x,0) > 0\} \cap \partial^0 B_r(x,0)) > 0. 
    \forall r > 0 \right\}.
\]
\end{definition}
%
We start with a result about the regularity of low multiplicity points.
\begin{lemma}\label{lem: mulitplicity 1}
If $K\cap\{y=0\} \subset \{x : m(x) \leq 1\}$  
then $\bv \in \C^{1,1/2}(K\cap\{y\geq0\})$.
\end{lemma}
\begin{proof}
According to Remark \ref{rem: red to bdd}, we will show local regularity of the functions in 
$B^+_r(x_0,0)$, where $r$ is small and $m(x_0)\leq 1$. We have three possibilities.

\emph{Case 1: $m(x_0) = 0$.} in this case, $\bv|_{\partial^0 B_r(x_0,0)} \equiv 0$, and the result is standard.

\emph{Case 2: $m(x_0) = 1$ and $v_i(x_0,0) > 0$.} By continuity of $v_i$, we can assume that $v_i|_{\partial^0 B_r(x_0,0)} > 0$, while by the segregation condition $v_j|_{\partial^0 B_r(x_0,0)} \equiv 0$ for every $j \neq i$. Let $\hat\bv$ be as in Definition \ref{def: hat}. Since in this case $\hat v_i = v_i$ on $\partial^0 B_r(x_0,0)$, it follows from \eqref{eqn: segr prop} that
\[
    \begin{cases}
        - \Delta \hat v_i = 0 &\text{ in $B_r^+(x_0,0)$}\\
        \partial_\nu \hat v_i = f(x,0, \dots, \hat v_i, \dots, 0) &\text{ in $\partial^ 0B_r^+(x_0,0)$}.
    \end{cases}
\]
The regularity of $\hat v_i$ (and thus of $v_i$) follows by the well established regularity theory of the semilinear Steklov problem.

\emph{Case 3: $m(x_0) = 1$ and the non trivial function $v_i$ is such that $v_i(x_0,0) = 0$}. Also in this case we can assume $v_j|_{\partial^0 B_r(x_0,0)} \equiv 0$ for $j \neq i$ and, as before, $v_i = \hat v_i \geq 0$ on $\partial^0 B_r(x_0,0)$. By continuity of $v_i$, we can also assume that $f_i = 0$ in $\partial^0 B_r(x_0,0)$. It follows that
\[
    \begin{cases}
        - \Delta \hat v_i = 0 &\text{ in $B_r(x_0,0)$}\\
        \partial_\nu \hat v_i = 0 &\text{ in $\partial^ 0B_r(x_0,0) \cap \{\hat v_i|_{y = 0} > 0$\}}\\
        \partial_\nu \hat v_i \geq 0 &\text{ in $\partial^ 0B_r(x_0,0)$}.
    \end{cases}
\]
As a consequence, $\hat v_i$ is a solution to the zero thin obstacle problem, for which $\C^{1,1/2}$ regularity (up to the obstacle) has been obtained in \cite[Theorem 5]{atcaf}.
\end{proof}
\begin{remark}
The analogous of the previous lemma holds true also when $s \neq 1/2$, in which
case $\C^{1,s}$ regularity can be shown, as a consequence of \cite{css}.
\end{remark}
Now, for $X\in B$, we introduce the the Morrey quotient associated to $\bv$ as
\[
    \Phi(X,r) := \frac{1}{r^{N+1-2\eps}}\int\limits_{B_r(X)\cap B^+}\sum_{i=1}^{k} |\nabla v_i|^2 \; \de x \de y.
\]
It is well known that if $\Phi$ is uniformly bounded for any $X\in K\cap\{y\geq0\}$ and 
$r<\dist(K,\partial^+B^+)$, then $\bv$ is H\"older continuous of exponent $1-\eps$ in 
$K\cap\{y\geq0\}$. Thus the proof of Theorem \ref{thm: intro_lastholdk} is based on the 
contradictory assumption that, for some $\eps>0$, there is a sequence $\{(X_n, r_n)\}_n$ such 
that $X_n \in K$, $r_n > 0$ and
\begin{equation}\label{eqn: contr_ass}
    \lim_{n \to \infty} \Phi(X_n, r_n) = \infty.
\end{equation}
To reach a contradiction we will use the following technical lemma.
\begin{lemma}[{\cite[Lemma 8.2]{ctvVariational}}]\label{lem: technical ctv}
Let $\Omega \subset \R^{N+1}$ and $v \in H^1(\Omega)$ and let
\[
    \Phi(X,r) := \frac{1}{r^{N+1-2\eps}}\int_{B_r(X) \cap \Omega} |\nabla v|^2 \de x \de y.
\]
If $(X_n,r_n) \subset \overline{\Omega} \times \R^+$ is a sequence such that $\Phi(X_n,r_n) \to \infty$, then $r_n \to 0$ and
\begin{enumerate}
    \item there exists $\{r_n'\} \subset \R^+$ such that $\phi(X_n, r_n') \to \infty$ and
       \begin{equation}\label{eqn: trace bounded}
       \int_{\partial B_{r'_n}(X_n) \cap \Omega} |\nabla v|^2 \leq \frac{N+1-2\eps}{r'_n}   
       \int_{B_{r'_n}(X_n) \cap \Omega} |\nabla v|^2;
       \end{equation}
    \item if $A \subset \overline{\Omega}$ and
    \[
        \dist(X_n, A) \leq C r_n
    \]
    then there exists a sequence $\{(X'_n,r'_n)\}$ such that $\phi(X'_n,r'_n) \to \infty$ and $X'_n \in A$ for every $n$.
\end{enumerate}
\end{lemma}
\begin{proof}[Proof of Theorem \ref{thm: intro_lastholdk}]
Using the second point of Lemma \ref{lem: technical ctv}, together with Remark 
\ref{rem: red to bdd}, we can assume 
without loss of generality that the contradictory assumption \eqref{eqn: contr_ass} holds for 
$\partial^0B^+ \ni X_n =: (x_n,0)$, for every $n$. Lemma \ref{lem: technical ctv} also implies that 
$r_n \to 0$, and 
that we can assume estimate \eqref{eqn: trace bounded} to hold for any $n$, with 
$\Omega=\{y>0\}$. 
Furthermore, since $\bv \in H^1(B^+)$, the function $r \mapsto \Phi((x_n,0), r)$ is continuous for 
$r > 0$ and it is uniformly bounded for $r$ faraway from $0$: as a consequence we can assume that
\[
    \Phi((x_n,0), r) \leq C \Phi((x_n,0), r_n) \qquad \forall r_n < r < \dist(K, \partial^+ B^+),
\]
for some constat $C$. Finally, by Lemmas \ref{lem: mulitplicity 1} and \ref{lem: technical ctv} 
we can assume $m(x_n) \geq 2$ for $n$ large, and thus $f_i(\cdot,v_i) \equiv 0$ on 
$B^+_{r_n}((x_n,0))$. 

Let us introduce a sequence of scaled function $\bv_n$ defined as 
\[
    v_{i,n}(X) := \frac{1}{\Phi((x_n,0),r_n)^{1/2} r_n^{1-\eps}} v_i((x_n,0) + r_n X) \qquad \text{ for $X \in B$}.
\]
By assumptions, $\|\nabla \bv_n\|_{L^2(B^+)}=1$ for every $n$, and
\begin{equation}\label{eqn: morrey bounded}
    \frac{1}{r^{N+1-2\eps}} \int_{B^+_r} \sum_{i=1}^{k} |\nabla v_{i,n}|^2 \leq C \qquad \forall 1 <r < r_n^{-1} \dist(K, \partial^+ B^+).
\end{equation}
We divide the rest of the proof in a number of steps.

\emph{Step 1: also $\|\bv_n\|_{L^2(B^+)}$ is uniformly bounded.} We argue by contradiction, 
assuming that $\|\bv_n\|_{L^2(B^+)} \to \infty$. Letting
\[
    \bu_n := \|\bv_n\|_{L^2(B^+)}^{-1} \bv_n
\]
we have that $\|\bu_n\|_{L^2(B^+)} =1$, while $\|\nabla \bu_n\|_{L^2(B^+)} \to 0$: there exists 
$\mathbf{d} \in \R^k$ such that
\[
    \bu_n \to \mathbf{d} \qquad \text{in $H^1(B^+)$}.
\]
Using the segregation condition $v_{i,n} \cdot v_{j,n} |_{y=0} = 0$, which passes to the strong 
limit, we infer that only one among the constant $d_i$ may be non trivial, say $d_1 > 0$. But 
recalling that the even extension of $\hat v_{i,n}$ across $\{y=0\}$ is superharmonic, we find
\[
    \hat v_{1,n}(0) = 0 \implies \int_{B^+} \sum_{j \neq 1} \frac{a_{ij}}{a_{ji}} v_{j,n} \geq \int_{B^+} v_{1,n}
\]
a contradiction, passing to the strong limit in $H^1(B^+)$.

\emph{Step 2: the sequence $\bv_n$ admits a nontrivial weak limit $\bar \bv \in H^1_{\loc}
(\overline{\R^{N+1}_+})$.} From Step 1 and the uniform estimate \eqref{eqn: morrey bounded} 
we infer the weak convergence; let us show that $\bar \bv$ is non trivial. 
To this end, we recall that
\[
    \begin{cases}
        - \Delta v_{i,n} =0 &\text{in $B^+$}\\
        v_{i,n} \partial_\nu v_{i,n} \leq 0 &\text{on $\partial^0 B^+$}.
    \end{cases}
\]
Testing the equation against $v_{i,n}$ and summing over $i$, we have
\[
    \int_{B^+} |\nabla \bv_{n}|^2 \leq \int_{\partial^+ B^+} \sum_{i=1}^{k} v_{i,n} \partial_\nu v_{i,n} \leq  \left(\int_{\partial^+ B^+} |\bv_{n}|^2 \cdot \int_{\partial^+ B^+} |\nabla \bv_{n}|^2\right)^{1/2}.
\]
Were $\bar\bv$ trivial, the right hand side would go to zero thanks to the compact embedding of the 
trace operator and the uniform estimate \eqref{eqn: trace bounded}, which is scaling invariant. 
This would imply strong convergence, in contradiction with the fact that the $L^2$ norm of 
$\nabla\bv_n$ is equal to 1.

\emph{Step 3: $\bar \bv(x,0) \equiv 0$ on $\R^N$.}
Let us consider the sequence $\hat \bv_n$ (recall Definition \ref{def: hat}). From 
\eqref{eqn: segr prop} (in the case $s=1/2$), we know that the pair $(\hat v_{i}^+,\hat v_{i}^-)$ is 
made of two continuous, subharmonic, nonnegative  functions such that $\hat v_{i}^+ \cdot\hat 
v_{i}^- = 0$ in 
$\R^{N+1}$, with nonpositive normal derivative on $\R^N$. As a result, they satisfy the 
assumption of the 
Alt-Caffarelli-Friedmann monotonicity formula (Lemma \ref{thm: ACF_s} with $a=0$ and $\mu=1$), from 
which we obtain
\begin{multline*}
    \frac{1}{r^{N+1}} \int_{B_{r}^+(x_n,0)} |\nabla \hat v^+_i|^2 \de x \de y \cdot \frac{1}{r^{N+1}} \int_{B_{r_n}^+(x_n,0)} |\nabla \hat v_i^-|^2 \de x \de y \leq \\ \frac{1}{r^{2}} \int_{B_{r}^+(x_n,0)} \frac{|\nabla \hat v_i^+|^2}{|X-(x_n,0)|^{N-1}} \de x \de y \cdot \frac{1}{r^{2}} \int_{B_{r}^+(x_n,0)} \frac{|\nabla \hat v_i^-|^2}{|X-(x_n,0)|^{N-1}} \de x \de y\leq C,
\end{multline*}
that is
\[
    \frac{1}{r^{N+1-2\eps}} \int_{B_{r}^+(x_n,0)} |\nabla \hat v^+_i|^2 \de x \de y \cdot \frac{1}{r^{N+1-2\eps}} \int_{B_{r}^+(x_n,0)} |\nabla \hat v_i^-|^2 \leq Cr^{4\eps}.
\]
Hence at most one of the two Morrey quotients can be unbounded. Moreover, by the triangular inequality, the possibly unbounded one diverges at most at the same rate of $\Phi((x_n,0),r_n)$. Scaling to $(\hat v_{i,n}^+,\hat v_{i,n}^-)$ we can distinguish among three different cases:
\begin{itemize}
    \item both $\|\nabla \hat v_{i,n}^+\|_{L^2(B_r)}$ and $\|\nabla \hat v_{i,n}^-\|_{L^2(B_r)}$ are infinitesimal. In this situation, we have that there exists $c \geq 0$ such that $\hat v_{i,n} \to c$. Since by even extension
        \[
            \begin{cases}
                - \Delta \hat v_{i,n} \geq 0 &\text{in $B_r$}\\
                \hat v_{i,n}(0,0) = 0
            \end{cases} \implies \int_{B_r} \hat v_{i,n} \leq 0
        \]
        we have that $\hat v_{i,n} \to c \leq 0$;
    \item there exists $c>0$ such that $\|\nabla \hat v_{i,n}^+\|_{L^2(B_r)} \geq c > 0$ while $\|\nabla \hat v_{i,n}^-\|_{L^2(B_r)} \to 0$. Testing the equation
    \[
        \begin{cases}
            - \Delta \hat v_{i,n}^+ \leq 0 &\text{in $B_r^+$}\\
            \hat v_{i,n}^+ \partial_\nu \hat v_{i,n}^+ \leq 0 &\text{on $\partial^0 B_r^+$}
        \end{cases}
    \]
with $\hat v_{i,n}^+$ we obtain that in the limit $\hat v_{i,n}^+ \rightharpoonup\hat v_i \neq 0$, 
and thus $\hat v_{i,n}^- \to 0$ strongly in $H^1(B_r)$. Using again the superharmonicity of $\hat 
v_{i,n}$ as before, we conclude that $\hat v_{i,n} \to 0$, in contradiction with $\hat v_i \neq 0$;
    \item there exists $c>0$ such that $\|\nabla \hat v_{i,n}^-\|_{L^2(B_r)} \geq c > 0$ while $\|
    \nabla \hat v_{i,n}^+\|_{L^2(B_r)} \to 0$. Reasoning as in the previous case, we obtain that 
    $\hat v_{i,n}^+ \to 0$ strongly in $H^1(B_r)$, thus $\hat v_{i,n} \rightharpoonup  \hat v_{i} 
    \leq 0$.
\end{itemize}
In any case, $\hat v_{i,n}^+ \to 0$ and $\hat v_{i,n} \rightharpoonup \hat v_{i} \leq 0$ in 
$H^1_{\loc}(\R^{N+1}_+)$ for all $i$, implying in particular that $ v_{i,n}|_{y=0} \to 
\bar v_i|_{y=0} = 0$ in $H^{1/2}_{\loc}(\R^{N})$.

\emph{Conclusion.} If we extend $\bar \bv$ evenly across $\{y=0\}$, we obtain a $k$-tuple 
of harmonic functions defined on $\R^{N+1}$ for which $\Phi(0,r) \leq C$ for all $r \geq 1$. 
From the Morrey inequality, we have that for any $X \in \R^{N+1}$, $|X|\geq1$,
\[
    |\bar \bv(X) - \bar \bv(0)| \leq C |X|^{1-\frac{N+1}{2}} \|\nabla \bar \bv\|_{L^2(B_{2|X|})}.
\]
As a result, we have
\[
    |\bar \bv(X) - \bar \bv(0)| \leq C |X|^{1-\eps}
\]
for every $|X|\geq1$, in contradiction with the fact that $\bv$ is harmonic in $\R^{N+1}$ and non 
trivial, thanks to the classical Liouville theorem.
\end{proof}
\begin{remark}
More general nonlinearities should be addressable, using similar arguments as before, 
once a generalization of the Caffarelli-Jerison-Kenig almost monotonicity formula 
\cite{cafjerken} to this setting were available.
\end{remark}
\begin{remark}
The case $s \neq 1/2$ could follow as a generalization of the previous proof, if not for the fact 
that, at the moment, no exact Alt-Caffarelli-Friedman monotonicity formula is available, in this 
setting: one could only show, by Lemma \ref{thm: ACF_s} and \ref{lem: elementary estimate}, 
the $\C^{0,\alpha}$ continuity of the limiting profiles, for every $\alpha < 2s$ and 
$\alpha \leq \mu$.
\end{remark}

\bibliography{alex_bib_LV}
\bibliographystyle{abbrv}

\end{document}